\documentclass[letterpaper, 10 pt, conference]{ieeeconf}  

\IEEEoverridecommandlockouts                              

\usepackage[utf8]{inputenc}
\usepackage[T1]{fontenc}

\usepackage{graphics} %
\usepackage{amsmath} %
\usepackage{amsfonts}
\usepackage{amssymb}  %
\usepackage{mathrsfs}

\makeatletter
\let\NAT@parse\undefined
\makeatother

\let\ieeebibliography\thebibliography

\usepackage[numbers,sort&compress]{natbib}
\usepackage[colorlinks]{hyperref}
\usepackage[commentmarkup=uwave,final]{changes}

\renewcommand\thebibliography[1]{\ieeebibliography{#1}}

\usepackage[IEEE]{altthm-styles}

\newenvironment{ass-enumerate}{\renewcommand{\theenumi}{(\alph{enumi}}\begin{enumerate}}{\end{enumerate}}

\newcommand\mto{\rightrightarrows}
\newcommand\Def[1]{:#1}
\DeclareMathOperator{\domain}{dom}
\DeclareMathOperator{\graph}{graph}

\DeclareMathOperator*{\argmin}{arg\,min}
\DeclareMathOperator{\lipschitz}{lip}

\definechangesauthor[color=purple]{TC}
\definechangesauthor[color=violet]{IK}

\title{\LARGE \bf
Input-to-State Stability of Newton Methods for Generalized Equations in Nonlinear Optimization$^\star$
}

\author{Torbjørn Cunis$^{1,2}$ and Ilya Kolmanovsky$^{2}$
\thanks{$^\star$\added[id=IK]{This research is supported in part by AFOSR
grant number FA9550-20-1-0385.}}
\thanks{$^{1}$Institute of Flight Mechanics and Control,
        University of Stuttgart, 70569 Stuttgart, Germany.
        {\tt\small tcunis@ifr.uni-stuttgart.de}}%
\thanks{$^{2}$Department of Aerospace Engineering, University of Michigan,
        Ann Arbor, MI 48104, USA.
        {\tt\small \{tcunis | ilya\}@umich.edu}}%
}

\begin{document}

\maketitle
\thispagestyle{empty}
\pagestyle{empty}

\begin{abstract}

We show that Newton method{s} for generalized equations are input-to-state stable with respect to disturbances such as \added[id=IK]{due to} inexact {computations}. We then use this result to obtain convergence and robustness of a {multistep Newton-type} method {for multivariate generalized equations}. We demonstrate the usefulness of the result{s} with {other} applications {to} nonlinear optimization. In particular, we provide a new proof for (robust) local convergence of the augmented Lagrangian method.

\end{abstract}

\section{Introduction}
Generalized equations are set-valued inequalities 
\begin{align}
    \label{eq:generalized}
    f(z) + F(z) \ni 0
\end{align}
where $f$ is a function and $F$ is a set-valued map. In nonlinear optimization, generalized equations appear frequently as the Karush--Kuhn--Tucker (KKT) system of necessary conditions \cite{Giorgi2004} with $f$ being the gradient of the Lagrangian, $F$ the normal cone to the constraint set, and $z$ {aggregating} primal and dual decision variables. Optimization algorithms can often be interpreted as solving \eqref{eq:generalized} for its root $\bar z$, the optimal solution. A common technique is Newton's method for generalized equations, which yields the iteration
\begin{align}
    \label{eq:newton}
    f(z_k) + \nabla f(z_k) (z_{k+1} - z_k) + F(z_{k+1}) \ni 0
\end{align}
and which, when applied to the KKT system, is better known as sequential quadratic programming \cite{Izmailov2014}. If the gradient {of $f$} in \eqref{eq:newton} does not exist or is unknown, Newton's method can be extended to the broader class of quasi-Newton {and Josephy-Newton} methods which include projected gradient descent and sequential convex (linear) programming. 
Robust local convergence properties {of Newton methods} {have been studied} under metric regularity assumptions \cite{Coleman1984, Dokov1998, Dontchev2010, Dontchev2013}.

In recent years, 
properties of optimization algorithms {have been studied when interconnected} with dynamic systems {and used to generate control actions} \cite{Liao-McPherson2020a, Leung2021b, Belgioioso2021, Hauswirth2021, Skibik2023}. A common framework here is \added[id=IK]{the one of} input-to-state stability (ISS) which combines concepts of robust stability and asymptotic gains with dissipation theory \cite{Sontag2004}. Previously, ISS was proven for \added[id=TC]{classical iterative methods for linear equations \cite{Hasan2013,Colabufo2020},} gradient descent \cite{Sontag2022}, and proximal gradient descent \cite{Cunis2023ifac}. In addition, the result {in} \cite{Dontchev2010} {on} stability of the sequence generated by \eqref{eq:newton} 
can be considered {as an ISS-like} result. {On the other hand, previous works on perturbed Newton methods for generalized equations such as \cite{Dokov1998, Dontchev2010} treated the input as a static deviation of the \highlight[id=IK]{limit point}.} Establishing ISS of Newton method{s} for generalized equations {enables the treatment of dynamic and time-varying perturbations, which are common in, e.g., the} analysis of interconnected optimization algorithms and optimization-based feedback.

The contributions of our paper are threefold: 
First{ly}, we formally prove {local} {ISS} of Newton method{s} \added[id=TC]{for generalized equations} {in the presence of} \added[id=TC]{generic} disturbances \added[id=TC]{including} \added[id=IK]{due to} inexact {computations} or erroneous gradients. 
Second{ly}, we propose a {multistep Newton-type} method for multivariate generalized equations{, which allows for lower-dimensional partial updates,} and prove its robust local convergence using the ISS property. 
We then demonstrate that the result of \cite{Dontchev2010} follows immediately from ISS. 
Thirdly, we apply our framework and ISS results to {approximate sequential programming} and the augmented Lagrangian method. 

\begin{table*}
    \centering
    \caption{{Josephy}-Newton methods to solve \eqref{eq:kkt}.}
    \label{tab:quasi-newton}
    \begin{tabular}{| p{.1\textwidth} | p{.3\textwidth} | p{.4\textwidth} |}
        \hline
        Algorithm & Choice of $H(z_k)$ & Interpretation \\
        \hline
        Sequential Quadratic Programming 
            & $\begin{pmatrix} \nabla^2 (h(x_k) + \langle g(x_k), y_k \rangle) & \nabla g(x_k)^* \\ \nabla g(x_k) & 0 \end{pmatrix}$ 
            & $(x,y)_{k+1}$ is the primal-dual solution to a quadratic program with linear constraints.
        \\
        \hline
        Sequential Convexification 
            & $\begin{pmatrix} 0 & \nabla g(x_k)^* \\ \nabla g(x_k) & 0 \end{pmatrix}$ 
            & $(x,y)_{k+1}$ is the primal-dual solution to a linear program. 
        \\
        \hline
        Projected Gradient Descent 
            & $\alpha^{-1} \mathbb I$ with $\alpha > 0$ 
            & $x_{k+1}$ is the projection of $\big(x_k - \alpha \nabla h(x_k) - \alpha g(x_k)^* y_k\big)$ onto $C$; and $y_{k+1} = y_k - \alpha g(x_k)$.
        \\
        \hline
    \end{tabular}
\end{table*}

\comment[id=TC]{Check if Lipschitz continuity in $(x,p)$ is sufficient for $\widehat {\operatorname {lip}}_x (\phi - f_H) = 0$.}

\section{Preliminaries}
If not noted otherwise, all spaces are considered either finite-dimensional or complete (Banach) vector spaces with norm $\|\cdot\|$. 
A set-valued map $F$ between vector spaces $X$ and $Y$, denoted $F: X \mto Y$, takes values $F(x) \subset Y$ for any $x \in X$. We define the domain and graph of $F$ as $\domain F = \{ x \in X \, | \, F(x) \neq \varnothing \}$ and $\graph F = \{ (x, y) \, | \, y \in F(x) \}$, respectively. For {a} closed convex set $C \subset X$, the normal cone mapping is $N_C(\bar x) = \{ y \in X^* \, | \, \forall x \in C, \langle y, x - \bar x \rangle \leq 0 \}$ if $\bar x \in C$ and $N_C(\bar x) = \varnothing$ else{, where $X^*$ is the dual space of $X$}. {The gradient of a function $f: X \to Y$ at $\bar x \in X$, if existing, is a linear operator $\nabla f(\bar x): X \to Y$; and we will assume that any gradient, if existing, is Lipschitz continuous around $\bar x$.}

\subsection{Continuity \& Regularity}
A set-valued mapping $F: X \mto Y$ is said to be {\em Lipschitz continuous} on $D \subset X$ with constant $\kappa \geq 0$ if $D$ is nonempty, $F(x)$ is closed, and
\begin{align}
    \label{eq:lipschitz}
    F(x') \subset \{ y' \in Y \, | \, \exists y \in F(x), \| y - y' \| \leq \kappa \| x - x' \| \}
\end{align}
for all $x, x' \in D$. The condition \eqref{eq:lipschitz} reduces to the classical Lipschitz continuity of functions if $F$ {is single-valued} on $D$. 
Let $U \subset X$ and $V \subset Y$ be neighbourhoods of $(\bar x, \bar y) \in \graph F$; the mapping $F$ has the {\em isolated calmness} property at $\bar x$ for $\bar y$ {with constant $\kappa$} if $F(\bar x) \cap V = \{\bar y\}$ and $x \mapsto F(x) \cap V$ satisfies \eqref{eq:lipschitz} for $x = \bar x$ and all $x' \in U$. Moreover, a function $f: X \to Y$ {which is Lipschitz continuous in a neighbourhood of $\bar x$ with constant $\kappa$} is a {\em {(Lipschitz) continuous} single-valued localization} of $F$ at $\bar x$ {with constant $\kappa$} for $\bar y$ if $F(x) \cap V = \{f(x)\}$ for all $x \in U$. 

\pagebreak

We now define notions of regularity.

\begin{definition}
    Take $(\bar x, \bar y) \in \graph F$; the mapping $F$ is {\em strongly regular} at $\bar x$ for $\bar y$ with {constant $\kappa$} if and only if $F^{-1}$ has a \added[id=IK]{Lipschitz} {continuous} single-valued localization at $\bar y$ for $\bar x$ {with constant $\kappa$}.
\end{definition}

\begin{definition}
    Take $(\bar x, \bar y) \in \graph F$; the mapping $F$ is {\em strongly subregular} at $\bar x$ for $\bar y$ {with constant $\kappa$} if and only if $F^{-1}$ has the isolated calmness property at $\bar y$ for $\bar x$ {with constant $\kappa$}.
\end{definition}

Recall that $F^{-1}(y) = \{ x \in X \, | \, y \in F(x) \}$ for all $y \in Y$. \added[id=IK]{Thus,} we say that $F$ is strongly regular (subregular) with constant $\kappa \geq 0$ if $\kappa$ is the constant of the Lipschitz continuous localization (isolated calmness property) of $F^{-1}$.

\begin{proposition}
    Let $F: X \mto Y$ be strongly regular (subregular) at $\bar x$ for $\bar y$ with constant $\kappa \geq 0$ and $(\bar x, \bar y) \in \graph F$; if $g: X \to Y$ is Lipschitz continuous with constant $\mu \in [0, \kappa^{-1})$, then $(g + F)$ is strongly regular (subregular) at $\bar x$ for $g(\bar x) + \bar y$ with constant $\kappa/(1 - \kappa \mu)$.
\end{proposition}
\begin{proof}
    See \cite[Theorem{s}~8.6 and 12.2]{Dontchev2021}.
\end{proof}

We next give an interpretation of strong regularity in the context of nonlinear optimization.


\subsection{Nonlinear Optimization}
Consider {a} nonlinear program,
\begin{align}
    \label{eq:nonlinearopt}
    \min_{x \in C} h(x) \quad \text{subject to $g(x) = 0$}
\end{align}
with primal variable $x \in X$, cost function $h: X \to \mathbb R$, constraint $g: X \to Y$, and closed convex set $C \subset X$. {If} $\bar x$ is an optimal solution of \eqref{eq:nonlinearopt}, $\nabla h(\bar x)$ and $\nabla g(\bar x)$ exist, and a suitable constraint qualification such as {the} linear independence {constraint qualification} (LICQ) holds, then the KKT system
\begin{align}
    \label{eq:kkt}
    \begin{pmatrix}
        \nabla h(\bar x) + \nabla g(\bar x)^* \bar y \\
        g(\bar x)
    \end{pmatrix}
    + N_{C \times X^*}\big((\bar x, \bar y)\big) \ni 0
\end{align}
is satisfied for some dual variable $\bar y \in X^*$ \cite[see, e.g.,][Theorem~5.7]{Mordukhovich2006ii}. \replaced[id=TC]{Eq.~\eqref{eq:kkt} is a generalized equation in the form of \eqref{eq:generalized} with $z = (x,y) \in Z = X \times Y$, where the left-hand side}{The left-hand side of the generalized equation \eqref{eq:kkt}} {is} a set-valued map{ing} \added[id=TC]{due to the normal cone mapping}. {In finite dimensions, this mapping} is strongly regular at $(\bar x, \bar y)$ for $0$ if and only if LICQ holds (a fortiori, $\bar y$ is unique) and $\bar x$ is a strongly-stable stationary solution\footnote{See, e.g., \cite[Definition~2.7]{Grothey2001} for a definition \added[id=IK]{of a strongly-stable stationary solution}.} of \eqref{eq:nonlinearopt}
Optimization algorithms which rely on solving \eqref{eq:kkt}, such as Newton-type methods, typical require strong regularity to guarantee that the result is in fact a {locally} optimal solution of \eqref{eq:nonlinearopt}.

In {Section~\ref{sec:applications}}, we will consider a {disturbed} version of \eqref{eq:nonlinearopt} where $h(\cdot,v)$ and $g(\cdot,v)$ are Lipschitz continuous functions of $v \in V$. 
If the left-hand side of the {\em parametrized} KKT system \eqref{eq:kkt} is strongly regular at $\bar v \in V$, then its solution mapping \added[id=TC]{$S: V \mto X \times Y$} has a Lipschitz continuous single-valued localization; this is the result of Robinson's implicit function theorem \cite[Theorem~8.5]{Dontchev2021}, which we extend {to the case of multivariate mappings} in the appendix.

\subsection{Newton Methods}
To solve the generalized equation \eqref{eq:generalized} with $f: Z \to Z'$ and $F: Z \mto Z'$, we define the iteration
\begin{align}
    \label{eq:quasi-newton}
    z_{k+1} \in z_k - \big(H(z_k) + F\big)^{-1} f(z_k)
\end{align}
where, broadly speaking, $H(z_k): Z \to Z'$ {helps to} approximate $f$ around $z_k$. Table~\ref{tab:quasi-newton} reviews some common choices of $H(z_k)$ for the KKT system \eqref{eq:kkt} with $z = (x, y)$ and the resulting optimization algorithms. Eq.~\eqref{eq:quasi-newton} can be interpreted as {a} generalized equation parametrized in the previous solution $z_k$.
In general, the sequences generated by \eqref{eq:newton} or \eqref{eq:quasi-newton} are not unique, nor is a solution guaranteed to exist. Under regularity assumptions, however, a sequence exists and converges to a root of \eqref{eq:generalized}.

\begin{theorem}[{\cite[Theorem~15.2]{Dontchev2021}}]
    Let $\bar z$ be a solution to \eqref{eq:generalized} such that $\nabla f(\bar z)$ exists. {If} $f + F$ is strongly subregular at $\bar z$ for 0, then for any $z_0$ sufficiently close to $\bar z$ there exists a sequence $\{z_k\}_{k=0}^\infty$ generated by \eqref{eq:newton} which converges quadratically to $\bar z$. 
    Moreover, if $f + F$ is strongly regular, {then} this sequence is unique.
    $\lhd$
\end{theorem}

In fact, any sequence that stays sufficiently close to $\bar z$ converges quadratically.
The following result provides sufficient conditions for convergence if $H(z_k)$ is not a derivative of $f$; we define $f_H(z, \zeta) = f(\zeta) + H(\zeta)(z - \zeta)$ and assume that $f_H(\cdot, \zeta)$ is Lipschitz continuous uniformly\footnote{We say that a function $f(x,y)$ is Lipschitz continuous with respect to $x$ {\em uniformly} in $y$ at $(\bar x, \bar y)$ if $f(\cdot, y)$ is Lipschitz continuous at $\bar x$ with some constant $\kappa < \infty$ for all $y$ {in a neighbourhood of} $\bar y$.} in $\zeta$ at $(\bar z, \bar z)$.

\begin{proposition}
    Let $\bar z$ be a solution to \eqref{eq:generalized} such that $f_H(x, \cdot)$ is Lipschitz continuous with constant $\gamma$ uniformly in $x$ at $(\bar z, \bar z)$; if $f_H(\cdot, \bar z) + F$ is strongly subregular at $\bar z$ for $0$ with constant $\kappa$ and $\kappa \gamma < 1$, then for any $z_0$ sufficiently close to $\bar z$ there exists a sequence generated by \eqref{eq:quasi-newton} which converges linearly to $\bar z$.
    Moreover, if $f_H(\cdot, \bar z) + F$ is strongly regular, the sequence is unique.
\end{proposition}
\begin{proof}
    This is a consequence of \cite[Theorems~12.4 and 8.5]{Dontchev2021} with $h = f_H(\cdot, \bar z)$.
\end{proof}

If $f$ is continuously differentiable at $\bar z$ and $H(z) = \nabla f(z)$, then $f_H(\cdot, \bar z)$ is the linearization of $f$ around $\bar z$ and $f_H(x, \cdot)$ is Lipschitz continuous uniformly in $x$ with constant $0$.

\subsection{Input-to-state Stability}
We {now consider} the robustness of the sequences generated by \eqref{eq:newton} or \eqref{eq:quasi-newton} under disturbances. To that extent, we consider {a} disturbed dynamic system
\begin{align}
    \label{eq:system}
    z_{k+1} = \Phi(z_k, v_k)
\end{align}
for all $k \in \mathbb N$, where $\mathbf v = (v_0, v_1, \ldots) \subset V$ is a sequence of disturbances with $\| \mathbf v \|_\infty \Def= \sup_{k \in \mathbb N} \| v_k \| < \infty$.
\added[id=TC]{The following definition makes use of the comparison function classes $\mathcal {KL}$ and $\mathcal K_\infty$ of monotonic functions; see \cite{Kellett2014} for details.}

\begin{definition}
    The system \eqref{eq:system} is {\em locally input-to-state stable} around $\bar z$ if and only if there exist $\epsilon, \delta > 0$ and functions $\beta \in \mathcal {KL}$ and $\gamma \in \mathcal K_\infty$ such that any sequence $\{z_k\}_{k=0}^\infty$ generated under disturbance $\| \mathbf v \|_\infty < \delta$ satisfies
    \begin{align*}
        \| z_k - \bar z \| \leq \beta(\| z_0 - \bar z \|, k) + \gamma(\| \mathbf v \|_\infty)
    \end{align*}
    for all $k \in \mathbb N$, provided that $\| z_0 - \bar z \| < \epsilon$.
\end{definition}

{The definition implies that the solution of \eqref{eq:system} converges to a ball $\mathcal B_{\gamma,\mathbf v}(\bar z)$ around $\bar z$ with radius given by \added[id=TC]{the gain} $\gamma(\| \mathbf v \|_\infty)$.}
The system \eqref{eq:system} is locally input-to-state stable around $\bar z$ if (and only {if}) there exists a continuous, positive definite function $V: Z \to \mathbb R_{\geq 0}$, constants $\epsilon > 0$ and $\delta > 0$, {and functions $\alpha, \gamma \in \mathcal K_\infty$ such that $\alpha < \operatorname{id}$ and \cite{Jiang2001}}
\begin{align}
    \label{eq:iss-newton}
    V(\Phi(z, v)) \leq \alpha V(z) + \gamma \| v \|
\end{align}
for all $(z, v) \in Z \times V$ with $\| z - \bar z \| < \epsilon$ and $\| v \| < \delta$.

\section{Main Results}
We prove local input-to-state stability of perturbed Newton method{s in the form of \eqref{eq:quasi-newton}, also referred to as} Josephy-Newton method, and propose a new multistep Newton-type method for multivariate generalized equations. A discussion of related results concludes this section.

\subsection{Josephy-Newton Method}
Our first result concerns the perturbed Josephy-Newton method,
\begin{align}
    \label{eq:perturbed-newton}
    f(z_k, v_k) + H(z_k,v_k) (z_{k+1} - z_k) + F(z_{k+1}) \ni 0
\end{align}
with $f: Z \times V \to Z'$ and $H(z,v): Z \to Z'$. As before, define $f_H(z,\zeta,v) = f(\zeta,v) + H(\zeta,v) (z - \zeta)$. \added[id=TC]{Here, the disturbance $v_k \in V$ might model, e.g., the inexact evaluation of the gradient $\nabla f(z)$ or a nonzero remainder in solving \eqref{eq:generalized}.} We make the following assumptions.

\begin{assumption}
    \label{ass:perturbed-newton}
    Let $\bar z \in Z$ and $\kappa, \gamma_z, \gamma_v > 0$ satisfy:
    \begin{ass-enumerate}
        \item $\bar z$ is a solution of $f(\cdot, 0) + F \ni 0$;
        \item $f_H(\cdot, \zeta, v)$ is Lipschitz continuous uniformly in $(\zeta, v)$ at $(\bar z, \bar z, 0)$;
        \item $f_H(z, \cdot, v)$ is Lipschitz continuous with constant $\gamma_z$ uniformly in $(z, v)$ at $(\bar z, \bar z, 0)$;
        \item $f_H(z, \zeta, \cdot)$ is Lipschitz continuous with constant $\gamma_v$ uniformly in $(z, \zeta)$ at $(\bar z, \bar z, 0)$;
        \item $f_H(\cdot, \bar z, 0) + F$ is strongly regular with constant $\kappa$ at $\bar z$ for 0;
    \end{ass-enumerate}
    and $\kappa \gamma_z < 1$.
\end{assumption}

Our result is based on an extension of Robinson's implicit function theorem for generalized equations with multiple parameters, given in the appendix.

\begin{theorem}
    \label{thm:perturbed-newton}
    Under Assumption~\ref{ass:perturbed-newton}, the iteration in \eqref{eq:perturbed-newton} is unique for $(z_k,v_k)$ sufficiently close to $(\bar z, 0)$ and satisfies
    \begin{align*}
        \|z_{k+1} - \bar z\| \leq \kappa\gamma_z \|z_k - \bar z\| + \gamma \|v_k\|
    \end{align*}
    that is, \eqref{eq:perturbed-newton} is locally input-to-state stable \added[id=TC]{around $\bar z$}.
\end{theorem}
\begin{proof}
    By virtue of Corollary~\ref{cor:generalized-multiple} with $h = f_H(\cdot,\bar z,0)$, the Josephy-Newton step \eqref{eq:perturbed-newton} has a locally unique solution $s: Z \times V \to Z$ for $(z_k,v_k)$ in the neighbourhood of $(\bar z, 0)$ satisfying
    \begin{align*}
        \| s(z,v) - s(z',v') \| \leq \kappa\gamma_z \|z - z'\| + \kappa\gamma_v \|v - v'\|
    \end{align*}
    for $(z,v), (z',v')$ around $(\bar z, 0)$. Taking $\alpha = \kappa\gamma_z < 1$ and $\gamma = \kappa\gamma_v$ as well as noting that $z_{k+1} = s(z_k,v_k)$ and $\bar z = s(\bar z,0)$ leads to the desired result.
\end{proof}

We can further strengthen this result if the disturbance affects the gradient of $f$, that is, $f(\cdot, v) = f$ and $H(\zeta,v) = \nabla f(\zeta) + v$, assuming continuous differentiability \added[id=IK]{of $f$}. Note that this implies Lipschitz continuity of $f_H$ with arbitrarily small constants $\gamma_z$ and $\gamma_v$; moreover, $f_H(\cdot,\bar z,0) + F$ is strongly regular if and only if $f + F$ is.
\added[id=TC]{In this case, \eqref{eq:perturbed-newton} resembles a {\em quasi-Newton} method. Specialising Theorem~\ref{thm:perturbed-newton} to the quasi-Newton method, we obtain quadratic convergence to $\mathcal B_{\gamma, \mathbf v}(\bar z)$ where the gain $\gamma$ vanishes close to $\bar z$.}

\begin{corollary}
    \label{cor:pertubed-gradient}
    Under Assumption~\ref{ass:perturbed-newton}, not only is 
    \begin{align}
        \label{eq:perturbed-gradient}
        f(z_k) + (\nabla f(z_k) + v_k)(z_{k+1} - z_k) + F(z_{k+1}) \ni 0
    \end{align}
    locally input-to-state stable {around $\bar z$} but the generated sequence \added[id=IK]{$\{z_k\}_{k = 0}^\infty$ satisfies}
    \begin{align}
        \label{eq:quadratic-convergence}
        \| z_{k+1} - \bar z \| \leq \kappa L \|z_k - \bar z\|^2 + \gamma_k \|v_k\|
    \end{align}
    \added[id=IK]{for all $k \in \mathbb N$}
    and $\gamma_k \to 0$ as $z_k$ approaches~$\bar z$.
\end{corollary}
\begin{proof}
    Let $\{z_k\}_{k = 0}^\infty$ be the sequence generated by \eqref{eq:perturbed-gradient} \added[id=TC]{which, by virtue of Theorem~\ref{thm:perturbed-newton}, exists}, is unique, and remains in the neighbourhood of $\bar z$ for $z_0$ and $\mathbf v$ close to $(\bar z,0)$. We \added[id=IK]{now argue} with \cite[Proof of Theorem~15.2]{Dontchev2021} that
    \begin{multline*}
        \| f(z_k) - f(\bar z) - (\nabla f(z_k) + v_k)(z_k - \bar z) \| \\ \leq L \|z_k - \bar z\|^2 + |\langle v_k, z_k - \bar z \rangle|
    \end{multline*}
    where $2L$ is the Lipschitz constant of $\nabla f(z)$ around $\bar z$, \added[id=IK]{guaranteed to exist by Assumption~\ref{ass:perturbed-newton} and $H(z,v) = \nabla f(z) + v$,} and hence \added[id=IK]{\eqref{eq:quadratic-convergence} holds}
    with $\gamma_k = \kappa \|z_k - \bar z\|$ by strong regularity of $f_H$.
\end{proof}

\subsection{Multistep Newton Method}
For our second result, we consider the perturbed multivariate generalized equation
\begin{align}
    \label{eq:multivariate}
    f(x,y,v) + F(x,y,v) \ni 0
\end{align}
with $f: X \times Y \times V \to Z'$ and $F: X \times Y \mto Z'$. We propose to solve \eqref{eq:multivariate} by the multistep Newton{-type} method
\begin{subequations}
    \label{eq:multivariate-newton}
\begin{align}
    \label{eq:multivariate-newton-inner}
    \tilde f(x_{k+1},y_k,v_k) + \tilde F(x_{k+1}\deleted[id=TC]{,y_k}) \ni 0 \\
    \label{eq:multivariate-newton-outer}
    f_{Hy}(x_{k+1},y_{k+1},y_k,v_k) + F(x_{k+1},y_{k+1}) \ni 0
\end{align}
\end{subequations}
where, \added[id=TC]{in the first step,} $\tilde f: X \times Y \times V \to Z''$ and $\tilde F: X \mto Z''$~pro\-vide a (possibly {lower-order}) generalized equation for $x$ parametrized in $y$; \added[id=TC]{in the second step,} $f_{Hy}(\xi,y,\eta,v) = f_H(\xi,\xi,y,\eta,v)$ with \[f_H: (x,\xi,y,\eta,v) \mapsto f(\xi,\eta,v) + H(\xi,\eta,v)(x - \xi, y - \eta)\] and operator $H(\xi,\eta,v): X \times Y \to Z'$ is a perturbed approximation of $f(\xi,y,0)$ with respect to $y$ around $\eta$. 
The \deleted[id=TC]{first} {inclusion} \added[id=IK]{\eqref{eq:multivariate-newton-inner}} could be solved inexactly, e.g., through a finite number of Newton-type steps, with error reflected by the disturbance. {The multistep Newton-type method is a useful framework to study bilevel optimization problems, where \eqref{eq:multivariate-newton-inner} corresponds to the KKT system of a lower-level parametrized optimization problem. We will demonstrate this on the example of the augmented Lagrangian method which can be interpreted as solving the dual problem of \eqref{eq:nonlinearopt}, \added[id=IK]{which leads to} a bilevel optimization \cite{Bertsekas1982a}.}

\begin{assumption}
    \label{ass:multivariate-newton}
    Let $\bar z = (\bar x, \bar y) \in X \times Y$ and $\tilde\kappa, \kappa, \gamma_y, \gamma_v, \allowbreak \gamma_w > 0$ satisfy:
    \begin{ass-enumerate}
        \item $(\bar x, \bar y)$ is a solution of \eqref{eq:multivariate};
        \item $\tilde f(\cdot, \bar y, 0) + \tilde F$ is strongly subregular with constant $\tilde\kappa$ at $\bar x$ for $0$; 
        \item $\tilde f(\cdot, y, v)$ is Lipschitz continuous uniformly in $(y,v)$ at $(\bar x, \bar y, 0)$;
        \item $\tilde f(x, \cdot, \cdot)$ is Lipschitz continuous with constant $\gamma_w$ uniformly in $x$ at $(\bar x, \bar y, 0)$;
        \item $f_H(\cdot, \xi, \cdot, \eta, v)$ is Lipschitz continuous uniformly in $(\xi,\eta,v)$ at $(\bar x, \bar x, \bar y, \bar y, 0)$;
        \item $f_H(x, \cdot, y, \cdot, v)$ is Lipschitz continuous with constant $\gamma_y$ uniformly in $(x,y,v)$ at $(\bar x, \bar x, \bar y, \bar y, 0)$;
        \item $f_H(x, \xi, y, \eta, \cdot)$ is Lipschitz continuous with constant $\gamma_v$ uniformly in $(x,\xi,y,\eta)$ at $(\bar x, \bar x, \bar y, \bar y, 0)$;
        \item $f_H(\cdot,\bar x,\cdot,\bar y,0) + F$ is strongly regular with constant $\kappa$ at $(\bar x, \bar y)$ for $0$; 
    \end{ass-enumerate}
    and $\kappa \gamma_y < 1$.
\end{assumption}

\begin{remark}
    {The mapping} $f_H(\cdot,\bar x,\cdot,\bar y,0) + F$ is strongly regular if $f(\cdot,\cdot,0) + F$ is strongly regular with constant $\tilde\kappa$ and $f_H(\cdot,\xi,\cdot,\eta,v) - f(\cdot,\cdot,0)$ is Lipschitz continuous with constant smaller than $\tilde\kappa^{-1}$ uniformly in $(\xi,\eta,v)$ \cite[Theorem~8.6]{Dontchev2021}.
\end{remark}
    
\begin{remark}
    \label{rem:multivariate-solution}
    An immediate consequence of Assumption~\ref{ass:multivariate-newton} is that the solution map{ping} $S: Y \times V \mto X$ of \eqref{eq:multivariate-newton-inner} has the isolated calmness property with constant $\tilde\kappa \gamma_w$ at $(\bar y,0)$ for $\bar x$ by virtue of \cite[Theorem~12.4]{Dontchev2021}.
\end{remark}
\begin{remark}
    If $F$ is a piecewise polyhedral mapping, then strong subregularity is equivalent to $\bar x$ being an isolated point in $S(\bar y, 0)$, a consequence of outer Lipschitz continuity of piecewise polyhedral mappings \cite[Theorem~12.5]{Dontchev2021}, which is again equivalent to a unique solution of \eqref{eq:multivariate}. 
\end{remark}

For the following result, we impose the norm on $X \times Y$ as $\|(x,y)\| \Def= \|x\| + \|y\|$.

\begin{theorem}
    \label{thm:multivariate-newton}
    Under Assumption~\ref{ass:multivariate-newton}, there exists a sequence $\{(x_k,y_k)\}_{k = 0}^\infty$ generated by \eqref{eq:multivariate-newton} for $(x_0,y_0)$ and $\mathbf v$ sufficiently close to $(\bar x, \bar y, 0)$ such that $\{y_k\}$ is unique and \eqref{eq:multivariate-newton-outer} is locally input-to-state stable with gain $\kappa \gamma_v$.
\end{theorem}

\begin{proof}
    Take $(y_k,v_k) \in Y \times V$ close to $(\bar y, 0)$; by strong subregularity (Assumption~\ref{ass:multivariate-newton}-b), there exists a solution $x_{k+1}$ of \eqref{eq:multivariate-newton-inner} close to $\bar x$. Let $y_{k+1}$ solve \eqref{eq:multivariate-newton-outer} and observe that
    \begin{multline*}
        0 \in
        f(x_{k+1},y_k,v_k) + F(x_{k+1},y_{k+1}) \\ 
        + {H(x_{k+1},y_k,v_k)(x_{k+1} - x_{k+1}, y_{k+1} - y_k)}
    \end{multline*}
    in other words, $(x_{k+1},y_{k+1})$ is a Josephy-Newton step in the sense of \eqref{eq:perturbed-newton} for \eqref{eq:multivariate} with $z_k = (x_{k+1},y_k)$. By virtue of Theorem~\ref{thm:perturbed-newton}, the point $y_{k+1}$ is unique and satisfies
    \begin{align*}
        \|(x_{k+1},y_{k+1}) - \bar z\| 
        &\leq \kappa \gamma_y \|(x_{k+1},y_k) - \bar z\| + \kappa \gamma_v \|v_k\| 
    \end{align*}
    hence, $\|y_{k+1} - \bar y\| \leq \kappa \gamma_y \|y_k - \bar y\| + \kappa \gamma_v \|v_k\|$
    by choice of $\|\cdot\|$ on $X \times Y$ and $\kappa \gamma_y < 1$ by Assumption~\ref{ass:multivariate-newton}.
    Moreover, the solution map $S$ of \eqref{eq:multivariate-newton-inner} has the isolated calmness property (Remark~\ref{rem:multivariate-solution}) and thus, {$\|x_{k+1} - \bar x\| \leq \tilde\kappa \gamma_w \|y_k - \bar y\| + \tilde\kappa \gamma_w \|v_k\|$}. 
    Combining these results we obtain
    \begin{align*}
        \|(x_k,y_k) - (\bar x,\bar y)\| \leq \alpha_k \|(x_0,y_0)\| + \gamma_\infty \|\mathbf v\|_\infty
    \end{align*}
    with $\alpha_k \Def= (\kappa \gamma_y)^{k-1}(\kappa \gamma_y + {\tilde\kappa \gamma_w})$,
    the desired result.
\end{proof}

\begin{remark}
    \label{rem:multivariate-operator}
It should be noted that the partial operator $H_x$ is never used but in the theoretical analysis and hence can freely be chosen 
to satisfy the strong regularity condition in Assumption~\ref{ass:multivariate-newton}. 
In particular, a possible choice for $H$ is
\begin{multline*}
    H(\xi,\eta,v): (d_x,d_y) \mapsto f(\xi+d_x,\eta,v) \\ + H_y(\xi+d_x,\eta,v)d_y - f(\xi,\eta,v)
\end{multline*}
with $H_y(\xi,\eta,v): Y \to Z'$, that is,
\begin{align*}
    f_H(x,\xi,y,\eta,v) \equiv f(x,\eta,v) + H_y(x,\eta,v)(y - \eta)
\end{align*}
and regularity and continuity of $f_H$ depend on $f_{Hy}$ only.
\end{remark}

We present applications of these results in nonlinear optimization in the next section.

\subsection{Related Results}
Previous works studied a Newton-type iteration of the form of \eqref{eq:perturbed-newton} with $H(z,p) = \nabla_x f(z,p)$ to solve parametrized generalized equations, assuming Fréchet differentiability of $f$ with respect to $z$ and continuity of $f$ and $\nabla_x f$. Under strong regularity assumptions similar to Assumption~\ref{ass:perturbed-newton}, the authors \added[id=IK]{of} \cite{Dokov1998} concluded that the sequence $\{z_k\}_{k=0}^\infty$ is locally unique and convergent to a solution $z(p)$ for any constant $p$ sufficiently close to $0$, and $\| z(p) - z(0) \| \leq \mu \|p\|$ for some constant $\mu > 0$. Furthermore, in \cite{Dontchev2010}, it was proven that the sequence satisfies
\begin{align*}
    \sup_{k \in \mathbb N_+} \|z_k - \bar z\| \leq \alpha \|z_0 - \bar z\| + \gamma \| p \|
\end{align*}
for some $\alpha < 1$ and $\gamma > 0$; this result is both necessary and sufficient for local input-to-state stability {in the sense of \eqref{eq:iss-newton}}.

Another classical topic in the study of Newton-type methods is the convergence of the iteration \eqref{eq:newton} or \eqref{eq:quasi-newton} if the right-hand side is a nonzero {remainder}, viz.
\begin{align*}
    f(z_k) + H(z_k)(z_{k+1} - z_k) + F(z_k) \ni e_k
\end{align*}
\replaced[id=IK]{typically corresponding to solving inexactly the underlying}{often studied in the context of solving inexact} linear equations (see, e.g., \cite{Dembo1982}). Using local input-to-state stability \added[id=IK]{properties}, we \added[id=IK]{can immediately} retrieve the desired convergence of $\{z_k\}_{k=0}^\infty$ to $\bar z$ if $\|e_k\| \to 0$.

\section{Applications}
\label{sec:applications}
We apply the results of Theorems~\ref{thm:perturbed-newton} and \ref{thm:multivariate-newton} to derive new robust convergence properties for nonlinear optimization algorithms. {We assume that $f$ and $g$ in \eqref{eq:nonlinearopt} are continuously differentiable in $x$ at $(\bar x,0)$ uniformly in $v$ and Lipschitz continuous in $v$ uniformly in $x$.}

\subsection{Approximate Sequential Quadratic Programming}
A classical approach to sequential quadratic programming is the approximation of the Hessian of the Lagrangian $L(x,y) \Def= h(x) + \langle g(x), y \rangle$ for \eqref{eq:nonlinearopt}, which appears in the upper-left block of the gradient when computing the (exact) Newton step for \eqref{eq:kkt}. Approximating the Hessian by a positive definite matrix $B_{k+1}$ at step $k \in \mathbb N$, the Newton step then becomes equivalent to solving the quadratic program \cite[Theorem~11.1]{Dontchev2021}
\begin{subequations}
    \label{eq:approx-sqp}
\begin{align}
    &\!\min_{x \in C} \frac{1}{2} \langle B_{k+1} (x - x_k), x - x_k \rangle + \nabla h(x_k)(x - x_k) \\
    &\text{subject to $g(x_k) + \nabla g(x_k)(x - x_k) = 0$} 
\end{align}
\end{subequations}
and taking $\added[id=IK]{z_{k+1} =} (x,y)_{k+1}$ as (unique) primal-dual solution of \eqref{eq:approx-sqp}.
Popular algorithms to compute the approximation $B_{k+1}$ along the solution $\{(x,y)_k\}_{k \in \mathbb N}$ include the BFGS and DFP methods (named, respectively, for its discoverers), which belong to the larger Broyden class of Hessian update formulas and often provide superlinear convergence of the quasi-Newton iteration \cite{Nocedal2006}. 

\begin{assumption}
    \label{ass:approx-sqp}
    Eq.~\eqref{eq:nonlinearopt} has an optimal solution $(\bar x, \bar y) \in X \times Y$ such that \eqref{eq:kkt} is strongly regular at $(\bar x, \bar y)$ for $0$; the update $B_{k+1} = \Psi(B_k,z_{k+1})$ is locally input-to-state stable around $\nabla^2 L(\bar x, \bar y)$.
\end{assumption}

{Hessian approximations such as BFGS and DFP often require additional conditions to ensure that $B_k \to \nabla^2 L(\bar x,\bar y)$. Here, however, we neglect} the intricacies of the approximation {and instead focus} on the interplay between quasi-Newton step and Hessian update.

\begin{proposition}
     Under Assumption~\ref{ass:approx-sqp}, the quasi-Newton step of \eqref{eq:approx-sqp} with Hessian update $B_{k+1} = \Psi(B_k,z_{k+1})$ is locally asymptotically stable. 
\end{proposition}
\begin{proof}
    Note that the KKT system of \eqref{eq:approx-sqp} can be written in the form of \eqref{eq:perturbed-newton} with $z_k = (x_k,y_k)$,
    \begin{align*}
        H(z_k,v_k) = \begin{pmatrix}
            \nabla^2 L(x_k, y_k) + v_k & \nabla g(x_k)^* \\ \nabla g(x_k) & 0
        \end{pmatrix}
    \end{align*}
    and $v_k = B_{k+1} - \nabla^2 L(x_k, y_k)$. By virtue of Corollary~\ref{cor:pertubed-gradient} and Assumption~\ref{ass:approx-sqp}, we have that
    \begin{align*}
        \| z_{k+1} - \bar z \| \leq \alpha_1 \| z_k - \bar z \| + \gamma \| v_k \|
    \end{align*}
    and 
    \begin{align*}
        \| v_{k+1} \| \leq \alpha_2 \| v_k \| + \gamma_B \| z_{k+1} - \bar z \|
    \end{align*}
    with $\alpha_1, \alpha_2 \in [0, 1)$, $\gamma, \gamma_B > 0$, and $\alpha_1, \gamma \to 0$ as $z_k \to \bar z$. Combining these results, we obtain
    \begin{align}
        \|z_{k+1} - \bar z \| + \|v_{k+1}\| \leq \bar \alpha (\|z_k - \bar z\| + \|v_k\|)
    \end{align}
    with $\bar \alpha = \max\{\alpha_1(1+\gamma_B), \alpha_2 + \gamma(1+\gamma_B)\}$; {assuming that $z_k$ is} sufficiently close to $\bar z$ such that $\bar \alpha < 1$ gives the desired result.
\end{proof}

This result can be {easily} extended to local ISS of approximated sequential programming for parametrized (perturbed) nonlinear programs.

\subsection{Augmented Lagrangian Method}
The augmented Lagrangian method solves the nonlinear program \eqref{eq:nonlinearopt} by iterating over
\begin{subequations}
    \label{eq:alm}
\begin{align}
    \label{eq:alm-primal}
    x_{k+1} &\in \argmin_{x \in C} \big\{ h(x,v_k) + \langle y_k, g(x,v_k) \rangle + \frac{\varrho}{2} \|g(x,v_k)\|^2 \big\} \\
    \label{eq:alm-dual}
    y_{k+1} &= y_k + \varrho g(x_{k+1},v_k)
\end{align}
\end{subequations}
for some penalty $\varrho > 0$ {and disturbance $v_k \in V$}. The cost function in \eqref{eq:alm-primal} is the titular {\em augmented Lagrangian}, parametrized in the dual variable $y_k$, and the necessary conditions can be written as {a} parametrized generalized equation
\begin{multline}
    \label{eq:alm-kkt-primal}
    \nabla h(x,v_k) + \nabla g(x,v_k)^* y_k \\ + \varrho \nabla g(x,v_k)^* g(x,v_k) + N_C(x) \ni 0
\end{multline}
provided that $h$ and $g$ are continuously differentiable. A classical result \cite{Bertsekas1982a} states that, under mild assumptions and for sufficiently large (but finite) value of $\varrho$, {the function minimized in} \eqref{eq:alm-primal} {with $v_k = 0$} becomes locally strictly convex and \eqref{eq:alm-dual} can be interpreted as gradient ascent for the dual problem. This \added[id=IK]{also} corresponds to strong regularity of $\eqref{eq:alm-kkt-primal}$ for all $y_k$ around $\bar y$.

\begin{assumption}
    \label{ass:alm}
    Eq.~\eqref{eq:nonlinearopt} {with $v = 0$} has an optimal solution $(\bar x, \bar y) \in X \times Y$ and \eqref{eq:kkt} is strongly regular at $(\bar x, \bar y)$ for $0$.
\end{assumption}

An immediate consequence is strong \added[id=TC]{sub}regularity of \eqref{eq:alm-kkt-primal} for sufficiently large penalties; {to that extent, we introduce}
\begin{align*}
    f_\varrho(x,y,y_k,v_k) = \begin{pmatrix}
        \nabla h(x,v_k) + \nabla g(x,v_k)^* y \\
        g(x,v_k) + \varrho^{-1} (y_k - y)
    \end{pmatrix}
\end{align*}
and study the augmented KKT system as follows.

\begin{lemma}
    \label{lem:alm}
    Under Assumption~\ref{ass:alm}, there exist {constants} $\varrho_0 > 0$ {and $k_{\varrho_0} > 0$} such that, {for all $\varrho \geq \varrho_0$,}
    \begin{ass-enumerate}
        \item $f_\varrho(\cdot,\cdot,\bar y,0) + N_{C \times X^*}$ is strongly \added[id=TC]{sub}regular at $(\bar x,\bar y)$ for $0$ {with constant $k_\varrho \in (0, k_{\varrho_0}]$};
        \item Eq.~\eqref{eq:alm-kkt-primal} is strongly \added[id=TC]{sub}regular at $\bar x$ for $0$ with {constant $k_\varrho \in (0, k_{\varrho_0}]$} if $y_k = \bar y$ {and $v_k = 0$}.
    \end{ass-enumerate}
\end{lemma}
\begin{proof}
    We observe that since \eqref{eq:kkt} is strongly \added[id=TC]{sub}regular at $(\bar x,\bar y)$ for $0$ with some constant $\kappa > 0$, the set-valued mapping
    \begin{align}
        \label{eq:kkt-penalty}
        F_\varrho(x,y) \Def=
        f_\varrho(x,y,\bar y,0)
        + N_{C \times X^*}((x,y))
    \end{align}
    is strongly \added[id=TC]{sub}regular at $(\bar x,\bar y)$ for $0$ with constant $k_\varrho = \varrho\kappa/(\varrho - \kappa)$ for all $\varrho > \kappa$ \cite[Theorem~12.2]{Dontchev2021}. Note that $k_\varrho$ is strictly decreasing as $\varrho \to \infty$. Substituting $y_x = \bar y + \varrho g(x,0)$, we have that $F_\varrho(x,y_x) \ni (\delta,0)$ if and only if
    \begin{align*}
        \nabla h(x,0) + \nabla g(x,0)^* \bar y + \varrho \nabla g(x,0)^* g(x,0) + N_C(x) \ni \delta 
    \end{align*}
    for all $\delta$ around $0$. Hence, \eqref{eq:alm-kkt-primal} is strongly \added[id=TC]{sub}regular at $\bar x$ for $0$ with constant $k_\varrho \leq k_{\varrho_0}$ if $y_k = \bar y$ and $\varrho \geq \varrho_0 > \kappa$.
\end{proof}

We show that the augmented Lagrangian method is an instance of the multistep Newton{-type} method \eqref{eq:multivariate-newton} for the generalized equation \eqref{eq:kkt} in $(x,y)$, hence proving local input-to-state stability. 
Note that our approach does not require $f$ to be twice differentiable.

\begin{proposition}
    Under Assumption~\ref{ass:alm}, the iteration~\eqref{eq:alm} is locally input-to-state stable around $(\bar x, \bar y)$ for all $\varrho \geq \bar \varrho > 0$.
\end{proposition}
\begin{proof}
    Let $(x_{k+1},y_{k+1})$ be the result of \eqref{eq:alm} {for a given $(y_k,v_k) \in Y \times V$}; then
    \begin{align}
        \label{eq:alm-newton}
        f_\varrho(x_{k+1},y_{k+1},y_k,v_k) + N_{C \times X^*}((x_{k+1},y_{k+1})) \ni 0
    \end{align}
    for any $\varrho > 0$.
    Eq.~\eqref{eq:alm-newton}
    corresponds to a partial Newton step for \eqref{eq:kkt} in the sense of \eqref{eq:multivariate-newton-outer} and Remark~\ref{rem:multivariate-operator}, where
    \begin{multline*}
        f_{Hy}(\xi,y,\eta,v) = \\
        \begin{pmatrix}
            \nabla h(\xi,v) + \nabla g(\xi,v)^* \eta \\ g(\xi,v)
        \end{pmatrix}
        +
        \begin{bmatrix}
            \nabla g(\xi,v)^* \\ -\varrho^{-1}
        \end{bmatrix}
        (y - \eta)
    \end{multline*}
    and $f_{Hy}(\xi,y,\cdot,v)$ is Lipschitz continuous with constant $\varrho^{-1}$ uniformly in $(\xi,y,v)$. Moreover, $f_{Hy}(\cdot,\cdot,\bar y,0) + N_{C \times X^*}$ is strongly \added[id=TC]{sub}regular with constant $k_\varrho \leq k_{\varrho_0}$ at $(\bar x,\bar y)$ for $0$ for all $\varrho \geq \varrho_0$ by virtue of Lemma~\ref{lem:alm}. Pick $\bar\varrho \geq \varrho_0$ such that $\bar\varrho^{-1} k_{\varrho_0} < 1$; the desired result follows from Theorem~\ref{thm:multivariate-newton} for any $\varrho \geq \bar\varrho$.
\end{proof}


\section{Conclusions}
{Newton methods for generalized equations play a major role in nonlinear optimization. Our local input-to-state stability result for the perturbed Josephy-Newton method enables the study of optimization algorithms interconnected with dynamic systems, such as in optimization-based control, under disturbed or uncertain conditions. In addition, our locally input-to-state stable multistep Newton-type method allows for advanced optimization techniques as demonstrated on the augmented Lagrangian method. Further work will focus on relaxations of strong regularity and Lipschitz continuity conditions within the general ISS framework.}

\addtolength{\textheight}{-14.75cm}   




\section*{Appendix}
We provide \deleted{an} implicit function theorem\added{s} for generalized equations with multiple parameters, extending \cite[Theorems~8.5 and 12.4]{Dontchev2021}. To that extent, define
\begin{align*}
    \widehat \lipschitz_x(f; (\bar x,\bar p)) = \limsup_{\substack{x_1,x_2 \to \bar x, x_1 \neq x_2 \\ p \to \bar p}} \frac{\|f(x_1,p) - f(x_2,p)\|}{\|x_1 - x_2\|}
\end{align*}
for $f: X \times P \to Y$, and note that $f(\cdot, p)$ is Lipschitz continuous with constant $\gamma$ uniformly in $p$ around $(\bar x, \bar p)$ if and only if $\widehat \lipschitz_x(f; (\bar x,\bar p)) \leq \gamma$.
We consider the parametrized generalized equation
\begin{align}
    \label{eq:generalized-multiple}
    f(x,p_1,p_2) + F(x) \ni 0
\end{align}
with solution map
\begin{align*}
    S: p = (p_1,p_2) \mapsto \{ x \in X \, | \, \text{$(x,p)$ solves \eqref{eq:generalized-multiple}} \}
\end{align*}
and $P = P_1 \times P_2$.

\begin{theorem}
    \label{thm:generalized-multiple}
    Let $(\bar x, \bar p) \in \graph S$ and suppose that $h: X \to Y$ satisfies
    \begin{ass-enumerate}
        \item $f(\bar x, \bar p) = h(\bar x)$;
        \item $h + F$ is strongly \added{sub}regular with constant $\kappa$ at $\bar x$ for~$0$;
        \item $f(\cdot,p) - h$ is Lipschitz continuous with constant $\mu$ uniformly in $p$ at $(\bar x,\bar p)$;
        \item $f(x,\cdot,p_2)$ and $f(x,p_1,\cdot)$ are Lipschitz continuous uniformly in $(x,p)$ at $(\bar x,\bar p)$;
    \end{ass-enumerate}
    and $\kappa \mu < 1$;
    then the solution $S(\cdot)$ of \eqref{eq:generalized-multiple} has \added{the isolated calmness property} at $\bar p$ for $\bar x$ satisfying
    \begin{multline*}
        \| x - \bar x \| \leq \omega \, \widehat \lipschitz_{p_1}(f; (\bar x, \bar p)) \| p_1 - \bar p_1 \| \\ + \omega \, \widehat \lipschitz_{p_2}(f; (\bar x, \bar p)) \| p_2 - \bar p_2 \|
    \end{multline*}
    with $\omega = (1 - \kappa \mu)^{-1} \kappa$ \added{for all $(p_1, p_2, x) \in \graph S$ in a neighbourhood of $(\bar x, \bar p)$}.
\end{theorem}
\begin{proof}
    The proof is analogous to \cite[Proof of Theorem~12.4]{Dontchev2021} using that
    \begin{align*}
        \| f(x,p_1,p_2) - f(x,\bar p_1,\bar p_2) \| \leq \gamma_1 \|p_1 - \bar p_1\| + \gamma_2 \|p_2 - \bar p_2\|
    \end{align*}
    for all $(x,p)$ around $(\bar x,\bar p)$ with some constants $\gamma_1, \gamma_2 \geq 0$ by uniform Lipschitz continuity\footnote{In fact, uniform calmness would suffice.} of $f$.
\end{proof}


\begin{corollary}
    \label{cor:generalized-multiple}
    If the assumptions of Theorem~\ref{thm:generalized-multiple} hold with $h + F$ being strongly regular with constant $\kappa$ at $\bar x$ for $0$,
    then the solution $S(\cdot)$ of \eqref{eq:generalized-multiple} has a single-valued localization $s: P_1 \times P_2 \to X$ at $\bar p$ for $\bar x$ satisfying
    \begin{align*}
        \widehat \lipschitz_{p_1}(s; \bar p) \leq \omega \, \widehat \lipschitz_{p_1}(f; (\bar x, \bar p)) \\
        \widehat \lipschitz_{p_2}(s; \bar p) \leq \omega \, \widehat \lipschitz_{p_2}(f; (\bar x, \bar p))
    \end{align*}
    with $\omega = (1 - \kappa \mu)^{-1} \kappa$.
    $\lhd$
\end{corollary}

\bibliographystyle{IEEEtran}
\bibliography{references}

\end{document}